\documentclass{amsart}

\usepackage[utf8]{inputenc}
\usepackage{amsmath,amsfonts,amsthm,amssymb,verbatim,etoolbox,color}
\usepackage{hyperref}

\usepackage[most]{tcolorbox}
\newcommand{\ind}[1]{1_{#1}}
\newcommand{\lo}[1]{\mathcal{L}(#1)}

\def\E{\mathbb{E}}
\newcommand{\bbf}{\mathbb{F}}

\newcommand{\bbz}{\mathbb{Z}}
\newcommand{\bbr}{\mathbb{R}}

\newcommand{\bbc}{\mathbb{C}}

\newcommand{\abs}[1]{\left\lvert #1\right\rvert}
\newcommand{\Abs}[1]{\lvert #1\rvert}
\newcommand{\brac}[1]{\left( #1\right)}

\newcommand{\norm}[1]{\left\lVert #1\right\rVert}
\newcommand*{\bbe}{
  \mathop{
    \mathchoice{\vcenter{\hbox{\larger[4]$\mathbb{E}$}}}
               {\kern0pt\mathbb{E}}
               {\kern0pt\mathbb{E}}
               {\kern0pt\mathbb{E}}
  }\displaylimits
}
\DeclareMathOperator{\rk}{rk}
\DeclareMathOperator{\supp}{supp}
\def\Ghat{\widehat{G}}
\def\fhat{\widehat{f}}
\def\zeroGhat{0_{\Ghat}}

\newtheorem{theorem}{Theorem}
\newtheorem{lemma}[theorem]{Lemma}

\newtheorem{corollary}[theorem]{Corollary}

\newtheorem{proposition}[theorem]{Proposition}
\theoremstyle{definition}
\newtheorem{definition}[theorem]{Definition}

\title[Kelley-Meka bounds for three-term progressions]{The Kelley-Meka bounds for sets free of three-term arithmetic progressions}

\author{Thomas F. Bloom}
\address{Mathematical Institute\\Woodstock Road\\Oxford\\OX2 6GG, United Kingdom}
\email{bloom@maths.ox.ac.uk}

\author{Olof Sisask}
\address{Department of Mathematics\\
     Stockholm University\\
     SE-106 91 Stockholm\\
     Sweden
}
\email{olof.sisask@math.su.se}

\begin{document}

\maketitle

\begin{abstract}
We give a self-contained exposition of the recent remarkable result of Kelley and Meka: if $A\subseteq \{1,\ldots,N\}$ has no non-trivial three-term arithmetic progressions then $\lvert A\rvert \leq \exp(-c(\log N)^{1/12})N,$ where $c>0$ is a constant.

Although our proof is identical to that of Kelley and Meka in all of the main ideas, we also incorporate some minor simplifications relating to Bohr sets. This eases some of the technical difficulties tackled by Kelley and Meka and widens the scope of their method. As a consequence, we improve the lower bounds for the problem of finding long arithmetic progressions in $A+A+A$, where $A\subseteq \{1,\ldots,N\}$.
\end{abstract}

How large can a subset of $\{1,\ldots,N\}$ be without containing a non-trivial three-term arithmetic progression $\{a,a+d,a+2d\}$? (Non-trivial here means that $d\neq 0$.) This seemingly innocuous question, asked by Erd\H{o}s and Tur\'an \cite{ErTu} in 1936, has led to a wealth of interesting mathematics, and has become one of the central questions in additive combinatorics. A large part of the reason for this is that the tools and techniques that have been developed to tackle it, starting with the argument of Roth \cite{Ro} from 1953, have turned out to be very influential, not only in dealing with other problems in additive number theory, but also in motivating the development of tools in other areas of mathematics -- particularly in harmonic analysis. 

This note gives an exposition of a recent remarkable breakthrough result of Kelley and Meka \cite{KM} concerning this question: they prove a very strong upper bound for the maximal size of sets free of three-term progressions, far smaller than any previously available.

Let $A\subseteq \{1,\ldots,N\}$ be a set which does not contain any (non-trivial) three-term arithmetic progressions. Even establishing that $\lvert A\rvert = o(N)$ is a difficult problem, this being Roth's landmark 1953 result \cite{Ro}. Since Roth's work there has been a sequence of quantitative improvements to the upper bound, all of the form $\lvert A\rvert \leq N/(\log N)^{C}$ for some constant $C>0$, culminating in recent work of the authors \cite{BS}, who showed that $C=1+c$ is permissible, where $c>0$ is some tiny constant (we refer to the introduction of \cite{BS} for further history). 

Kelley and Meka's new upper bound is of a whole new order of magnitude. 
\begin{theorem}[Kelley-Meka]\label{th-main-int}
If $A\subseteq \{1,\ldots,N\}$ contains no non-trivial three-term arithmetic progressions, then
\[\abs{A} \leq \frac{N}{\exp\left(c(\log N)^{1/12}\right)}\]
for some absolute constant $c>0$.
\end{theorem}

The Kelley-Meka bound is a huge leap forward,
 and tantalisingly close to the best possible such bound. Indeed, we know that there are subsets of $\{1,\ldots,N\}$ of size at least $\exp(-c(\log N)^{1/2})N$, where $c>0$ is some universal constant, that contain no non-trivial three-term progressions. This was first proved by Behrend \cite{Be} using a beautiful construction that uses lattice points on high-dimensional spheres; small improvements have also been established by Elkin \cite{El} and Green and Wolf \cite{GW}. Getting anywhere near Behrend-style bounds for this problem has been a long-standing goal in the area, and the argument of Kelley and Meka achieves this in a beautiful and elegant way.

Our reasons for writing this note are
\begin{enumerate} 
\item to explain the Kelley-Meka approach using terminology and a perspective perhaps more familiar to researchers in additive combinatorics, and
\item to provide some minor technical refinements which allow for a proof of the full Theorem~\ref{th-main-int} which involves only `classical' Bohr set techniques, rather than the ad hoc method employed in \cite{KM} (in particular we answer \cite[Question 6.2]{KM} of whether such an approach is possible in the affirmative). This widens the scope of potential applications over the integers, and allows for integer analogues of the other main results in \cite{KM}.
\end{enumerate}
It must be made clear, however, that our sole contribution is at the technical level, allowing the Bohr set machinery to run a little more smoothly; all of the main ideas are the same as in \cite{KM}.

\subsection*{The finite field case}
As usual in this area, a simpler model case is provided by replacing $\{1,\ldots,N\}$ with the vector space $\mathbb{F}_q^n$. 
We will present Kelley and Meka's ideas in this model setting before the general case, since the former is technically much simpler, while still containing all of the important new ideas. Kelley and Meka's argument in $\bbf_q^n$ establishes the following. 
\begin{theorem}[Kelley-Meka]\label{th-main-ff}
If $q$ is an odd prime and $A\subseteq \bbf_q^n$ has no non-trivial three-term progressions, then $\lvert A\rvert \leq q^{n-cn^{1/9}}$ for some constant $c>0$.
\end{theorem}
The utility of Theorem~\ref{th-main-ff} is as a demonstration of the proof techniques, since for the result itself a stronger bound of $\lvert A\rvert \leq q^{n-cn}$ is available via the polynomial method, as shown by Ellenberg and Gijswijt \cite{EG}. Unfortunately, however, there is no known analogue of the polynomial method for the integer problem, so achieving strong bounds for the integer problem via this method is out of reach. Kelley and Meka's proof of Theorem~\ref{th-main-ff} uses no polynomial methods, and instead uses techniques from probability and Fourier analysis, which can be generalised (using classical Bohr set machinery) to the integer setting.

After presenting the Kelley-Meka argument for $\bbf_q^n$ we will generalise this, using the language of Bohr sets, to prove Theorem~\ref{th-main-int}. Kelley and Meka take a different, more ad hoc route, iterating over high-dimensional progressions. This is an ingenious alternative, that may itself have further applications, but our aim here is to show that more classical existing techniques also suffice. As a consequence, we can also establish the following integer analogue of \cite[Corollary 1.12]{KM}, which finds large subspaces in $A+A+A$ for $A\subseteq \bbf_q^n$.

\begin{theorem}\label{th-3A}
If $A\subseteq \{1,\ldots,N\}$ has size $\alpha N$, then $A+A+A$ contains an arithmetic progression of length
\[\geq \exp(-C\log(2/\alpha)^3)N^{c/\log(2/\alpha)^{9}},\]
where $C,c>0$ are constants.
\end{theorem}
For comparison, the previously best bound known was $N^{\alpha^{1+o(1)}}$, originally due to Sanders \cite{Sa1}. A construction due to Freiman, Halberstam, and Ruzsa \cite{FHR} shows that no exponent better than $c/\log(2/\alpha)$ is possible.

Kelley and Meka deduce the above bounds for sets without three-term arithmetic progressions, and more besides, from a more general type of structure result. This says, roughly speaking, that for any reasonably large set $A$ inside a finite abelian group there exists some structured set $V$ (e.g. an affine subspace in the $\bbf_q^n$ case) such that $A$ restricted to $V$ is very `regular' in its additive behaviour -- that is, for `most' $x$ the number of solutions to $x=a+b$ with $a,b \in A \cap V$ is very close to the average number of solutions.

We will first state a precise version of this structural result for $\bbf_q^n$. It is convenient to introduce the notion of a normalised indicator function: if $A\subseteq G$ is non-empty with size $\lvert A\rvert=\alpha \lvert G\rvert$ then we write $\mu_A=\alpha^{-1}\ind{A}$.  (For a set $A$, we write $1_A$ for the indicator function of $A$, taking the value $1$ on $A$ and $0$ elsewhere.) If $V\subseteq G$ and $1\leq p<\infty$, then we denote the $L^p(V)$ norm of $f : G \to \bbc$ by
\[\norm{f}_{p(\mu_V)} = \brac{\frac{1}{\abs{V}} \sum_{x \in V} \abs{f(x)}^p}^{1/p}.\]
Note, for example, that our choice of normalisation is chosen to ensure $\norm{\mu_A}_{1(\mu_G)}=1$. 
We will measure the aforementioned regularity of $A\subseteq V$ by bounding the $L^p(V)$ norm of the difference between the convolution 
\[\mu_A\ast \mu_A(x) = \frac{1}{\lvert G\rvert}\sum_{a,b\in G}1_{a+b=x}\mu_A(a)\mu_A(b)=\frac{\lvert G\rvert}{\lvert A\rvert^2}\sum_{a,b\in A}1_{a+b=x}\]
and its expected value. An elementary calculation shows that, when $V\leq G$ is a subgroup and $A\subseteq V$,
\[\norm{\mu_A\ast \mu_A}_{1(\mu_V)}=\frac{\abs{G}}{\abs{V}}=\norm{\mu_V}_{1(\mu_V)}.\]
That is, the average of $\mu_{A}\ast \mu_{A}$ over $V$ agrees with the average of $\mu_V$ itself. 

Note carefully that, even though we are taking a \emph{local} $L^p$ norm restricted to $V$ we are keeping $\mu_A$ and $\ast$ normalised relative to the \emph{global} group $G$. This is a little confusing at first, but when we move from $\bbf_q^n$ to general groups it is much more convenient to keep as many definitions as possible global, rather than relativising to the local $V$, since in general (e.g. when we move from $\bbf_q^n$ to $\bbz/N\bbz$) the subgroup $V$ will be replaced with a set that is only approximately structured.

\begin{theorem}[Kelley-Meka]\label{th-ff-gen}
Let $\epsilon>0$. There is a constant $C=C(\epsilon)>0$ such that the following holds. Let $G=\bbf_q^n$ for some prime $q$ and $n\geq 1$.  Let $p\geq 1$. For any non-empty $A\subseteq G$ with $\abs{A}=\alpha \abs{G}$ there exists a subspace $V\leq G$ with
\[\mathrm{codim}(V)\leq Cp^4\log(2/\alpha)^5\]
and $x \in G$ such that, if $A' =(A-x)\cap V$, then
\begin{enumerate}
\item \label{item:dens1'} $\abs{A'}\geq (1-\epsilon)\alpha \abs{V}$ and
\item \label{item:small_moment'}
$\norm{\mu_{A'}\ast \mu_{A'}-\mu_V}_{p(\mu_V)} \leq \epsilon \frac{\abs{G}}{\abs{V}}$.
\end{enumerate}
\end{theorem}
The factor $\frac{\abs{G}}{\abs{V}}$ here can be thought of as a normalising/scaling factor, corresponding to the fact that the convolution is defined over $G$ rather than over $V$. 

Intuitively, the second item of the conclusion says, as $p\to \infty$, that $\mu_{A'}\ast \mu_{A'}=(1+O(\epsilon))\mu_V$ with high probability, as $x$ ranges uniformly over $V$. Therefore, when studying many problems involving the additive behaviour of $A'$, one can replace $A'$ with a random subset of $V$ of the same density. This is, of course, a very powerful tool. The cost is that we have had to restrict our original set $A$ to an affine subspace $x+V$ to find this regularity, and so having the codimension of $V$ be as small as possible is important for quantitative applications. For example, in the deduction of Theorem \ref{th-main-ff} from this one may take $p$ to be around $\log(2/\alpha)$ and $\epsilon$ a constant. The resulting codimension bound of $C\log(2/\alpha)^9$ then corresponds to the exponent $1/9$ in Theorem \ref{th-main-ff}.

\subsection*{The general group case: Bohr sets}
To prove Theorems \ref{th-main-int} and \ref{th-3A} we shall use a more general version of Theorem \ref{th-ff-gen}, applicable to any finite abelian group (which in applications will be $\mathbb{Z}/N\mathbb{Z}$). One difficulty in even writing down the appropriate statement is working out what should play the role of subspaces for general abelian groups. This is not obvious; fortunately for us, however, this was already done by Bourgain \cite{Bo}, who showed that a suitable generalisation of a subspace, for these purposes, is a \emph{Bohr set}. A Bohr set is an approximate level set of some characters, or more precisely, a set of the shape 
\[ B = \mathrm{Bohr}_\nu(\Gamma) = \left\{ x\in G : \abs{1-\gamma(x)}\leq \nu\textrm{ for all }\gamma\in\Gamma\right\}\]
for some $\nu\geq 0$ (known as the width) and some $\Gamma\subseteq \widehat{G}$ (known as the frequency set). Often the most important thing about the frequency set is its size $d=\abs{\Gamma}$, which is called the rank of the Bohr set. 

Background material on Bohr sets can be found in Appendix~\ref{app-bohr}, but the unfamiliar reader can for now think of the above Bohr set $B$ of rank $d = \abs{\Gamma}$ as roughly like the embedding in $G$ of the lattice points in a $d$-dimensional box in $\mathbb{\bbr}^d$ of side-length proportional to $\nu$. Writing $B_\rho$ for the same Bohr set but with the `width' $\nu$ replaced by $\rho\cdot \nu$, one then has the approximate closure property that $B + B_\rho \approx B$ provided $\rho$ is small -- in particular, $B+B_\rho$ is not much larger than $B$ provided $\rho$ is small compared to $1/d$. It turns out that this approximate additive closure property of Bohr sets of rank $d$ can, for the purposes of this paper, be used in place of the exact rigid structure provided by subspaces of codimension $d$ in $\mathbb{F}_q^n$ (which are indeed Bohr sets themselves, of rank $d$ and width $\nu = 0$).

Since Bohr sets only enjoy an approximate closure property, statements involving them necessarily require more technical conditions and quantitative overhead. This is why the use of the finite field model of $\bbf_q^n$ is invaluable in this area: an idea can be tested with subspaces in a relatively clean way, and only once the idea has been proven to have real quantitative strength does the work of translating the argument to work over general Bohr sets need to begin.

The following is the generalisation of Theorem~\ref{th-ff-gen} required for our applications. Kelley and Meka did not prove such a statement, but speculated that this should be possible in \cite[Footnote 9]{KM}. 

The reader should compare the statement to Theorem~\ref{th-ff-gen}, and at first reading may wish to pretend that $B = B' = B_\rho$ is the same subspace $V\leq \bbf_q^n$. For comparison to the conclusion of Theorem~\ref{th-ff-gen} it may help to note that, when $B$ is a subspace and $A'\subseteq B$, then $\mu_B\ast \mu_B=\mu_B=\mu_{A'}\ast \mu_B$, and so
\[(\mu_{A'}-\mu_B)\ast (\mu_{A'}-\mu_B)=\mu_{A'}\ast \mu_{A'}-\mu_B.\]

\begin{theorem}\label{th-int-gen}
There is a constant $c>0$ such that the following holds. Let $\delta,\epsilon\in (0,1)$, let $p \geq 1$ and let $k$ be a positive integer such that $(k,\abs{G})=1$. There is a constant $C=C(\epsilon,\delta,k)>0$ such that the following holds.

For any finite abelian group $G$ and any subset $A\subseteq G$ with $\abs{A}=\alpha \abs{G}$ there exists a regular Bohr set $B$ with
\[\rk(B)\leq Cp^4\log(2/\alpha)^5\]
and
\[\abs{B}\geq \exp\left(-Cp^5\log(2p/\alpha)\log(2/\alpha)^6\right)\abs{G}\]
and $A' \subseteq (A-x)\cap B$ for some $x\in G$ such that
\begin{enumerate}
\item \label{item:dens1} $\abs{A'}\geq (1-\epsilon)\alpha \abs{B}$,
\item \label{item:dens2} $\abs{A'\cap B'}\geq (1-\epsilon)\alpha\abs{B'}$, where $B'=B_{\rho}$ is a regular Bohr set with $\rho\in (\tfrac{1}{2},1)\cdot c\delta\alpha/dk$, and 
\item \label{item:small_moment}
\[\norm{(\mu_{A'}-\mu_B)\ast (\mu_{A'}- \mu_B)}_{p(\mu_{k\cdot B'})} \leq \epsilon\frac{\abs{G}}{\abs{B}}.\]
\end{enumerate}
\end{theorem}

In other words, for any dense set $A$, we can find a low-rank Bohr set $B$ such that the restriction of (a translate of) $A$ to $B$ has almost the same relative density and has its convolution extremely balanced: it is close to its average, as measured in an $L^p$-sense over another large Bohr set. This roughly means that, when studying additive problems involving this restriction, we can replace $A$ by a random set of the same density. Part \eqref{item:dens2} of the conclusion is there for somewhat technical reasons: since we need to work with both the Bohr set $B$ and a narrowed copy $B_\rho$, as described above, we want to know that $A'$ has large density on $B_\rho$ as well as on $B$.

To give an idea of the parameters: to deduce Theorem \ref{th-main-int}, we shall take $\epsilon$ and $\delta$ some small constants, $k=2$, and $p \asymp \log(2/\alpha)$. (The dependence on $\epsilon,\delta,k$ is not hard to track explicitly, but is a distraction for the present applications.) An identical proof gives a statement with the measure $\mu_{k\cdot B'}$ in part (3) replaced by $\mu_{k\cdot B'+t}$ for any $t\in B$, or indeed $\mu_B$ or $\mu_B\ast \mu_B$, the latter two of which may appear more natural. These are, however, slightly weaker, and certainly for the application to three-term arithmetic progressions it is important that the measure over which the $L^p$ norm is taken is supported on some suitably narrowed copy of $B$.

In Section~\ref{sec-sketch} we provide an informal overview of the main steps of the argument. In Section~\ref{sec-key} we prove what are, in our opinion, the most important steps of the argument in sufficient generality for the results over both $\mathbb{F}_q^n$ and $\{1,\ldots,N\}$. In Section~\ref{sec-ff} we make the overview more precise and provide full proofs for the $\mathbb{F}_q^n$ case. Finally, in Sections~\ref{sec-int1} and \ref{sec-int2} we show how this argument can be directly adapted for the integers using Bohr sets. We will proceed by proving Theorem~\ref{th-int-gen} first and then deducing Theorems \ref{th-main-int} and \ref{th-3A}.

\subsection*{Improvements} Although Kelley and Meka's breakthrough is close to the best possible bound, there is still a gap between the exponent $1/12$ of Theorem \ref{th-main-int} and the exponent $1/2$ in the Behrend example, and it is natural to ask whether further progress is possible. Indeed, a small modification of the method as presented in this paper allows for a slight improvement: the $1/12$ of Theorem~\ref{th-main-int} can be replaced by $1/9$ with a relatively clean argument (the only modification required is to the almost-periodicity part). A further tedious lengthy technical optimisation allows for an exponent of $5/41$. Since the focus of this paper is an exposition of the method and results of Kelley and Meka, we will detail these improvements in a separate forthcoming note. The same modification leads to an improvement of the exponent in Theorem~\ref{th-main-ff} from $1/9$ to $1/7$, and an improvement of the exponent in Theorem~\ref{th-3A} from $9$ to $7$.

We believe that an exponent of $1/7$ (in the statement of Theorem~\ref{th-main-int}) is the natural limit of these methods, in that achieving anything better will require significant new ideas. Of course, the Behrend exponent of $1/2$ would be the final target, but this seems quite far out of reach still; indeed, an exponent of $1/3$ (or perhaps even $1/4$) seems to be the limit of any argument that uses any sort of `density increment' argument with Bohr sets (whether using Kelley-Meka ideas or a more traditional Fourier analytic approach). 

\subsection*{Acknowledgements} The first author is supported by a Royal Society University Research Fellowship. We would like to thank Zander Kelley and Raghu Meka for generously sharing a copy of their preprint with us, and their encouragement in writing this exposition. We would also like to thank Ben Green for helpful conversations, Ilya Shkredov for showing us an alternative argument for the proof of Lemma~\ref{lem-key-spec}, and an anonymous referee for useful suggestions.
\subsection*{Notational conventions}
Logarithmic factors will appear often, and so in this paper we use the convenient abbreviation $\lo{\alpha}$ to denote $\log(2/\alpha)$. (The $2$ here is just a convenient device to make sure that $\lo{\alpha}\geq 1/2$, say, whenever $\alpha\in(0,1]$.)

In statements which refer to $G$, this can be taken to be any finite abelian group (although for the applications this will always be either $\mathbb{F}_q^n$ or $\mathbb{Z}/N\mathbb{Z}$). We use the normalised counting measure on $G$, so that 
\[\langle f,g\rangle = \bbe_{x\in G}f(x)\overline{g(x)}\textrm{ and }\norm{f}_p=\brac{\bbe_{x\in G}\abs{f(x)}^p}^{1/p}\textrm{ for }1\leq p<\infty,\]
where $\bbe_{x\in G}=\frac{1}{\lvert G\rvert}\sum_{x\in G}$. For any $f,g:G\to \bbc$ we define the convolution  and the difference convolution\footnote{We caution that, while convolution is commutative and associative, difference convolution is in general neither.} as
\[f\ast g(x)=\bbe_y f(y)g(x-y)\quad\textrm{and}\quad f\circ g(x)=\bbe_yf(x+y)\overline{g(y)}.\]
Note the useful adjoint property
\[\langle f, g\ast h\rangle = \langle f\circ h,g\rangle.\]
We furthermore write $f^{(p)}$ for the $p$-fold convolution $f^{(p)} = f*f*\cdots*f$, where there are $p$ copies of $f$.

For some purposes it is conceptually cleaner to work relative to other non-negative functions on $G$, so that if $\mu:G\to\bbr_{\geq 0}$ has $\norm{\mu}_1=1$ we write 
\[\langle f,g\rangle_\mu = \bbe_{x\in G}\mu(x)f(x)\overline{g(x)}\textrm{ and }\norm{f}_{p(\mu)}=\brac{\bbe_{x\in G}\mu(x)\abs{f(x)}^p}^{1/p}\textrm{ for }1\leq p<\infty.\]
(The special case above is the case when $\mu\equiv 1$.)

 We use the slight abuse of notation that if $\mu:G\to \bbr_{\geq 0}$ with $\norm{\mu}_1=1$, then $\mu(A)=\norm{\ind{A}}_{1(\mu)}$ is the density of $A$ relative to $\mu$. Unless specified otherwise, $\mu$ is the uniform measure on $G$. (So that, for example, $\mu(A)=\alpha=\abs{A}/\abs{G}$ is the density of $A$ within $G$.) We write $\mu_A=\alpha^{-1}\ind{A}$ for the normalised indicator function of $A$ (so that $\norm{\mu_A}_1=1$). We will sometimes speak of $A\subseteq B$ with relative density $\alpha=\abs{A}/\abs{B}$.

The Fourier transform of $f:G\to\bbr$ is $\widehat{f}:\widehat{G}\to \bbc$ defined for $\gamma\in\widehat{G}$ as
\[\widehat{f}(\gamma)=\bbe_{x\in G}f(x)\overline{\gamma(x)},\]
where $\widehat{G} = \{ \gamma : G \to \bbc^\times : \text{$\gamma$ a homomorphism} \}$ is the dual group of $G$. We will also use convolution of functions $f, g : \widehat{G} \to \bbc$, defined as $f\ast g(\gamma) = \sum_{\chi \in \widehat{G}} f(\chi)g({\gamma-\chi})$, and denote $k$-fold convolution again by $f^{(k)}$ for such functions.\footnote{Note that we are using additive notation for the group operation on $\widehat{G}$.}
We note the following elementary facts:
\begin{itemize}
\item $\widehat{f\ast g}=\widehat{f}\cdot \widehat{g}$ and $\widehat{f\circ f}=\Abs{\widehat{f}}^2$ (so in particular the Fourier transform of $f\circ f$ is a non-negative function on $\widehat{G}$),
\item $\bbe_x f(x)^k = \widehat{f}\ast\cdots \ast\widehat{f}(\zeroGhat)$, where the convolution is $k$-fold, and 
\item if $\mu:G\to\bbr_{\geq 0}$ has $\norm{\mu}_1=1$, then the Fourier transform of $\mu-1$ is $\widehat{\mu}\ind{\neq \zeroGhat}$.
\end{itemize}
Note that these three facts immediately imply that the Fourier transform of $\mu_A\circ \mu_A-1$ is non-negative, that $\bbe (\mu_A\circ \mu_A-1)^k \geq 0$ for any integer $k$, and (coupled with the triangle inequality) that $\norm{\mu_A\ast \mu_A-1}_p\leq \norm{\mu_A\circ \mu_A-1}_p$ when $p$ is an even integer. Although easily seen via the Fourier transform, the latter two facts also have purely `physical' proofs, as we will see later.

Finally, we use the Vinogradov notation $X \ll Y$ to mean $X = O(Y)$, that is, there exists some constant $C>0$ such that $\abs{X}\leq CY$. We write $X\asymp Y$ to mean $X\ll Y$ and $Y\ll X$, and $X=\Omega(Y)$ to mean $Y = O(X)$. An expression like $1+\Omega(1)$ thus means a quantity bounded below by an absolute constant strictly bigger than $1$. The appearance of parameters as subscripts indicates that the implied constant may depend on these parameters (in some unspecified fashion).

\section{Sketch of the argument}\label{sec-sketch}

In this section we provide a sketch of the Kelley-Meka proof of the finite field model case, Theorem~\ref{th-main-ff}, along with some personal commentary and context. Kelley and Meka also include some fascinating comparisons of this approach with earlier work and speculations about it and alternative approaches, and we encourage the reader to study Appendices A and B of \cite{KM} and consider the interesting questions therein (although note that we answer their Question A.4 in the affirmative below).

Let $A\subseteq \bbf_q^n$ be a set of density $\alpha$ and $C\subseteq \bbf_q^n$ a set of density $\gamma$. (Note that a three-term arithmetic progression is a solution to $x+y=2z$, and thus for their study we will choose
\[C=2\cdot A=\{ 2a : a\in A\},\]
in which case $\gamma=\alpha$.) How many solutions to $a_1+a_2=c$ with $a_1,a_2\in A$ and $c\in C$ do we expect? If $A$, $C$ are random sets, then we expect $\approx \alpha^2\gamma q^{2n}$ many such solutions, which is to say
\[\langle \mu_A\ast \mu_A,\mu_C\rangle \approx 1.\]
The Kelley-Meka approach begins with some constant discrepancy from this expected count, say $\langle \mu_A\ast \mu_A,\mu_C\rangle \leq 1/2$, and shows that this leads to a large density increment of $A$ on some subspace $V\leq \bbf_q^n$ with codimension $O(\lo{\alpha}^{O(1)})$ (that is, shows that there is such a subspace $V$ and a translate $A'$ of $A$ such that $\mu_V(A')\geq (1+\Omega(1))\alpha$, meaning that $A'$ has significantly larger density in $V$ than $A$ has in $\bbf_q^n$). This is done using mostly physical-based methods, rather than the Fourier-based methods that have dominated the study of three-term progressions thus far.

Our presentation of the Kelley-Meka strategy breaks it down into five key steps. 
\begin{tcolorbox}[enhanced,attach boxed title to top center={yshift=-3mm,yshifttext=-1mm},
colback=blue!5!white,colframe=blue!75!black,
  title=1 : H\"{o}lder lifting,fonttitle=\bfseries,
  boxed title style={size=small,colframe=red!50!black} ]
If $\langle \mu_A\ast \mu_A, \mu_{C}\rangle \leq 1/2$, then $\norm{\mu_A\circ \mu_A-1}_p\geq 1/4$ for some $p\ll \lo{\gamma}$.
\end{tcolorbox}
This is essentially a one-line application of H\"{o}lder's inequality:
\[\tfrac{1}{2}\leq \abs{\langle \mu_A\ast \mu_A-1,\mu_C\rangle} \leq \norm{\mu_A\ast \mu_A-1}_p\norm{\mu_C}_{p/(p-1)}\leq 2\norm{\mu_A\ast \mu_A-1}_p\]
for sufficiently large $p$, since $\norm{\mu_C}_{p/(p-1)}=\gamma^{-1/p}$, and noting that $\norm{\mu_A\ast \mu_A-1}_p\leq \norm{\mu_A\circ \mu_A-1}_p$ when $p$ is an even integer. Although trivial, the passage from few three-term arithmetic progressions to large $L^p$ norm of $\mu_A\circ \mu_A-1$ was rarely used in previous work, most of which begins with the Fourier deduction that $\sum_{\gamma\neq 0} \abs{\widehat{\mu_A}(\gamma)}^2\abs{\widehat{\mu_{C}}(\gamma)}\gg 1$.

This physical H\"{o}lder step was used (along with almost-periodicity) in \cite{BSlog} to achieve density bounds of $\alpha\leq (\log N)^{-1+o(1)}$. Indeed, in a sense the approach of \cite{BSlog} corresponds to carrying out Steps 1, 4, and 5 of the present sketch. The advancement of \cite{BS} past the logarithmic density barrier couples this with delicate structural information on the Fourier side (following seminal ideas of Bateman and Katz \cite{BK}). 

It is incredible that the following two steps, which are simple enough to prove in only a couple of pages, perform far better quantitatively than this elaborate Fourier-side approach.
\begin{tcolorbox}[enhanced,attach boxed title to top center={yshift=-3mm,yshifttext=-1mm},
colback=blue!5!white,colframe=blue!75!black,
  title=2 : Unbalancing,fonttitle=\bfseries,
  boxed title style={size=small,colframe=red!50!black} ]For any $f:G\to \bbr$ such that $\widehat{f}\geq 0$, if $\norm{f}_p\geq 1/4$, then $\norm{f+1}_{p'}\geq 1+1/8$ for some $p'\ll p$. In particular, if $\norm{\mu_A\circ \mu_A-1}_p\geq 1/4$, then $\norm{\mu_A\circ \mu_A}_{p'}\geq 1+1/8$.
\end{tcolorbox}
This step is essential to the success of the Kelly-Meka argument, and rests on the fact that the Fourier transform of $f$ is non-negative. (Note that it is not true for an arbitrary function.) While the spectral non-negativity of $\mu_A\circ \mu_A-1$ has played a role in some arguments before (e.g. in the spectral boosting aspect of \cite{BS}), to our knowledge it has not before been so cleanly expressed, and its potential had not been fully appreciated within additive combinatorics.

We remark that, as pointed out to the authors by Shkredov, an entirely physical proof of this step is possible, with no mention of the Fourier transform at all, if we replace the assumption $\widehat{f}\geq 0$ with $f=g\circ g$ for some function $g:\bbr\to \bbc$ (note that this is indeed satisfied in our application since $\mu_A\circ \mu_A-1=(\mu_A-1)\circ (\mu_A-1)$). We refer to the proof of Lemma~\ref{lem-key-spec} for details.

At this point we digress to note that, instead of following Steps 1 and 2 as Kelley and Meka do, one could obtain the conclusion $\norm{\mu_A\circ \mu_A}_p\geq 1+\Omega(1)$ for some $p\ll \lo{\gamma}$ from the assumption that $\langle \mu_A\ast \mu_A,\mu_C\rangle \leq 1/2$ using more classical Fourier-based methods. (In particular this observation answers \cite[Question A.4]{KM} in the affirmative.) By converting the inner product to Fourier space and applying the triangle inequality we observe that
\[1/2\leq \sum_{\lambda\neq 0}\abs{\widehat{\mu_A}(\lambda)}^2\abs{\widehat{\mu_C}(\lambda)}.\]
It follows that, for some choice of signs $c_\lambda\in\bbc$, we have
\[1+1/2\leq \sum_\lambda \abs{\widehat{\mu_A}(\lambda)}^2\abs{\widehat{\mu_{C}}(\lambda)}=\bbe_{x\in C}\sum_{\lambda}c_\lambda \abs{\widehat{\mu_A}(\lambda)}^2\lambda(-x).\]
Applying H\"{o}lder's inequality to the left-hand side and using orthogonality of characters yields, for any even integer $p$,
\[1+1/2\leq \gamma^{-1/p}\brac{\sum_{\lambda_1,\ldots,\lambda_{p}}c_{\lambda_1}\cdots \overline{c_{\lambda_{p}}}\abs{\widehat{\mu_A}(\lambda_1)}^2\cdots \abs{\widehat{\mu_A}(\lambda_{p})}^21_{\lambda_1+\cdots-\lambda_{p}=0}}^{1/p}.\]
We can discard the signs $c_\lambda$ by the triangle inequality, and then by orthogonality the sum here is in fact equal to $\norm{\mu_A\circ \mu_A}_p$, and we are done choosing $p$ suitably large. This sort of step has already played a major role in previous Fourier-based approaches to Roth's theorem, although generally with the $\widehat{\mu_A}$ restricted to some `large spectrum', where it then yields information about the additive relations within this large spectrum. A striking feature of the work of Kelley and Meka is that this information is far more useful on the physical side.

We end our digression here and return to the sketch.

\begin{tcolorbox}[enhanced,attach boxed title to top center={yshift=-3mm,yshifttext=-1mm},
colback=blue!5!white,colframe=blue!75!black,
  title=3 : Dependent random choice,fonttitle=\bfseries,
  boxed title style={size=small,colframe=red!50!black} ]If $\norm{\mu_A\circ \mu_A}_p\geq 1+1/8$, then there are $A_1,A_2\subseteq A$ of density at least $\alpha^{O(p)}$ such that 
\[\langle \mu_{A_1}\circ \mu_{A_2}, \ind{S}\rangle \geq 1-1/32\]
where $S=\{ x : \mu_A\circ \mu_A(x) > 1+1/16\}$.
\end{tcolorbox}
This is, in our opinion, the most important step (although of course every step is necessary, and in particular the previous unbalancing step is crucial to obtaining the information required for this step). In combination with the previous two steps, we have now converted the original three variable deficiency information of $\langle \mu_A\ast \mu_A,\mu_C\rangle\leq 1/2$ into four variable abundancy information $\langle \mu_{A_1}\circ \mu_{A_2},\mu_A\circ \mu_A\rangle \geq 1+\Omega(1)$. This is very promising, since Schoen and the second author \cite{SS} have already shown that almost-periodicity can prove quasi-polynomial bounds (that is, of the same shape as the Kelley-Meka bound) for 4 variable equations. Another indication of the strength of the conclusion obtained in Step 3 is that Sanders \cite{Sa2} has shown (also using almost-periodicity) that $S$ contains a large proportion of a large subspace with codimension $O(\lo{\alpha}^{O(1)})$.

The proof of this step, called sifting in \cite{KM}, is completely elementary, and uses dependent random choice -- one takes $A_i$ to be the intersection of $p$ randomly chosen translates of $A$. A simple expectation calculation combined with the $L^p$ information then verifies that there must exist some choice of translates which satisfy both the density conditions and the inner product condition. 

Arguments of this kind have appeared before in additive combinatorics, dating back in some form to Gowers' proof of Szemer\'{e}di's theorem \cite{Go}. For example, when $p=2$ this method was used by Schoen \cite{Sc} to prove strong bounds for the Balog-Szemer\'{e}di-Gowers theorem, and similar manipulations for larger $p$ have played an extensive role in work of Schoen and Shkredov (see for example \cite{ScSh1,Sh}). A very similar statement also appears in work of Sanders \cite[Lemma 1.9]{Sa4}, itself a generalisation of an argument of Gowers \cite[Lemma 11]{Go}.

Despite this previous work, the true potential of this method (when coupled with the powerful technique of almost-periodicity) in applications to the study of three-term progressions and related problems had been overlooked before the breakthrough of Kelley and Meka.

\begin{tcolorbox}[enhanced,attach boxed title to top center={yshift=-3mm,yshifttext=-1mm},
colback=blue!5!white,colframe=blue!75!black,
  title=4 : Almost periodicity,fonttitle=\bfseries,
  boxed title style={size=small,colframe=red!50!black} ]For any sets $A_1,A_2,S\subseteq\bbf_q^n$, if $A_1,A_2$ have density at least $\alpha$, then there is a subspace $V$ with codimension $O(\lo{\alpha}^4)$ such that
\[\abs{\langle \mu_V\ast \mu_{A_1}\circ \mu_{A_2},\ind{S}\rangle - \langle \mu_{A_1}\circ \mu_{A_2},\ind{S}\rangle }\leq 1/100.\]
\end{tcolorbox}

Almost-periodicity statements of this type have played an important role in additive combinatorics since their introduction by Croot and the second author \cite{CS}, most notably in the work of Sanders achieving quasi-polynomial bounds for inverse sumset theorems \cite{Sa2,Sa3}. The conventional wisdom was that, despite their success in inverse sumset problems and for translation invariant equations in four or more variables (see \cite{SS} and the earlier \cite{ScSh2} for six or more variables), they were not able to achieve significant results for three-term arithmetic progressions. (Although the argument of \cite{BS} did make fundamental use of almost-periodicity, most of the work in that paper was Fourier-analytic.) Kelley and Meka have dispelled this illusion completely.

 It should be noted, however, that there is no novelty in the actual form of almost-periodicity used by Kelley and Meka -- the new strength is a result of the context in which they use it.

\begin{tcolorbox}[enhanced,attach boxed title to top center={yshift=-3mm,yshifttext=-1mm},
colback=blue!5!white,colframe=blue!75!black,
  title=5 : Density increment,fonttitle=\bfseries,
  boxed title style={size=small,colframe=red!50!black} ]If $\langle \mu_A\ast \mu_A,\mu_{C}\rangle \leq 1/2$, then there is an affine subspace $V$ of codimension $O(\lo{\alpha}^4\lo{\gamma}^4)$ on which $A$ has density at least $(1+\tfrac{1}{100})\alpha$.
\end{tcolorbox}

This is just a trivial combination of the previous 4 steps (noting that the density increment condition can also be phrased as $\norm{\mu_A\ast \mu_V}_\infty \geq 1+1/100$). This density increment condition can now be iteratively applied to eventually obtain a lower bound for $\langle \mu_{A'}\ast \mu_{A'},\mu_C\rangle$ (with $A'$ now perhaps some subset of a translate of $A$).

For example, the deduction of Theorem~\ref{th-main-ff} is routine with $C=2\cdot A$. Indeed, a density increment such as $\alpha\mapsto (1+\Omega(1))\alpha$ can occur at most $O(\lo{\alpha})$ many times, after which we must halt with many three-term arithmetic progressions found in the intersection of $A$ with some affine subspace, the codimension of which is bounded above by $O(\lo{\alpha}^9)$. 

Alternatively, carrying through these steps with $C$ being the complement of $A+A$ leads to the following.

\begin{theorem}[Kelley-Meka]
If $A\subseteq \bbf_q^n$ has density $\alpha$ and $\gamma\in(0,1]$, then there is some affine subspace $V\leq \bbf_q^n$ of codimension $O(\lo{\alpha}^5\lo{\gamma}^4)$ such that 
\[\abs{(A+A)\cap V}\geq (1-\gamma)\abs{V}.\]
\end{theorem}
For comparison, the best bound previously available, due to Sanders \cite{Sa2}, had codimension $O(\lo{\alpha}^4\gamma^{-2})$. The latter is slightly better when $\gamma$ is constant (as was the case of interest in \cite{Sa2}) but much weaker when $\gamma\approx \alpha$, which is the regime of interest for three-term arithmetic progressions.

\section{The key new lemmas}\label{sec-key}

In this section we prove general forms of Steps 2 and 3 that will be used for both the $\mathbb{F}_q^n$ and integer case.

\subsection{Unbalancing of spectrally non-negative functions}

The following precise form of Step 2 will suffice for all our applications.
\begin{lemma}\label{lem-key-spec}
Let $\epsilon\in(0,1)$ and $\nu:G\to\bbr_{\geq 0}$ satisfy $\norm{\nu}_1=1$ and $\widehat{\nu}\geq 0$. Let $f:G\to\bbr$ be such that $\widehat{f}\geq 0$. (Or, alternatively, assume that $f=g\circ g$ and $\nu=h\circ h$ for some $g,h:G\to\bbc$ .)

If $\norm{f}_{p(\nu)}\geq \epsilon$ for some $p\geq 1$, then $\norm{f+1}_{p'(\nu)}\geq 1+\tfrac{1}{2}\epsilon$ for some $p'\ll_\epsilon p$.
\end{lemma}
We have left the dependence on $\epsilon$ unspecified since our applications only use $\epsilon\gg 1$. The proof below delivers $p'\ll \epsilon^{-1}\log(\epsilon^{-1})p$. Kelley and Meka use a different method (see \cite[Appendix D]{KM}) which is more efficient, giving $p'\ll \epsilon^{-1}p$, but requires use of an external (though simple) fact about the binomial distribution.
\begin{proof}
We first establish the important fact that, for any integer $k\geq 1$,
\begin{equation}\label{eq-nonneg}\langle \nu, f^k\rangle \geq 0.
\end{equation}
We will give two proofs of \eqref{eq-nonneg}, depending on whether the Fourier assumption $\widehat{f},\widehat{\nu}\geq 0$ or the physical assumption $f=g\circ g$ and $\nu=h\circ h$ is used. Given \eqref{eq-nonneg} no further use of the Fourier transform is required, and the proofs converge.

The Fourier proof of \eqref{eq-nonneg}, which is used by Kelley and Meka, is immediate from Parseval's identity:
\[\langle \nu,f^{k}\rangle=\langle \widehat{\nu},\widehat{f}^{(k)}\rangle,\]
where we recall that $f^{(k)} = f*f*\cdots*f$ denotes the $k$-fold convolution, and the right-hand side is non-negative since $\widehat{f},\widehat{\nu}\geq 0$.

We now present the alternative physical proof of \eqref{eq-nonneg}, assuming $f=g\circ g$ and $\nu=h\circ h$. This argument was pointed out to the authors by Shkredov, who has made extensive use of the following kind of manipulations (for example in \cite{Sh}). We observe that
\begin{align*}
\langle \nu, f^k\rangle
&=
\bbe_x h\circ h(x)g\circ g(x)^k\\
&= \bbe_{y_1,y_2}h(y_1)\overline{h(y_2)}\brac{\bbe_z g(y_1+z)\overline{g(y_2+z)}}^k\\
&=\bbe_{z_1,\ldots,z_k}\bbe_{y_1,y_2}h(y_1)\overline{h(y_2)}g(y_1+z_1)\cdots \overline{g(y_2+z_k)}\\
&= \bbe_{z_1,\ldots,z_k}\abs{\bbe_y h(y)g(y+z_1)\cdots g(y+z_k)}^2,
\end{align*}
which is clearly non-negative as each summand is.

We now show how \eqref{eq-nonneg} implies the conclusion. Without loss of generality we can assume that $p\geq 5$ is an odd integer. Using \eqref{eq-nonneg}, since $2\max(x,0)=x+\abs{x}$ for $x\in\bbr$ and $f^{p-1}=\abs{f}^{p-1}$, we have 
\[2\langle \max(f,0),f^{p-1}\rangle_\nu=\langle \nu,f^p\rangle+\langle \abs{f},f^{p-1}\rangle_\nu\geq \norm{f}_{p(\nu)}^p\geq \epsilon^p.\]
Therefore, if $P=\{ x : f(x) \geq 0\}$, then $\langle \ind{P},f^{p}\rangle_\nu\geq \frac{1}{2}\epsilon^{p}$. Furthermore, if $T=\{ x\in P : f(x) \geq \tfrac{3}{4}\epsilon\}$, then $\langle \ind{P\backslash T},f^p\rangle_\nu <(\tfrac{3}{4}\epsilon)^p\leq \tfrac{1}{4}\epsilon^p$, and hence by the Cauchy-Schwarz inequality
\[\nu(T)^{1/2}\norm{f}_{2p(\nu)}^{p}\geq \langle \ind{T}, f^{p}\rangle_\nu \geq \tfrac{1}{4}\epsilon^{p}.\]
On the other hand, by the triangle inequality
\[\norm{f}_{2p(\nu)}\leq 1+\norm{f+1}_{2p(\nu)}\leq 3,\]
or else we are done, with $p'=2p$. Hence $\nu(T)\geq (\epsilon/10)^{2p}$. It follows that, for any $p'\geq 1$,
\[\norm{f+1}_{p'(\nu)}\geq \langle \ind{T},\abs{f+1}^{p'}\rangle_\nu^{1/p'}\geq  (1+\tfrac{3}{4}\epsilon)(\epsilon/10)^{2p/p'}.\]
The desired bound now follows if we choose $p'$ a sufficiently large multiple (depending on $\epsilon$) of $p$.
\end{proof}

\subsection{An application of dependent random choice}
We now use dependent random choice (or what Kelley and Meka call `sifting') to prove a general form of Step 3. This makes use of (a generalisation of) the fact that
\[ \norm{ 1_A \circ 1_A }_p^p = \bbe_{s_1,\ldots,s_{p} \in G} \mu((A+s_1) \cap \cdots \cap (A+s_{p}))^2 \] 
to convert $L^p$-information about a convolution to information about the nested intersections appearing in the right-hand side. This identity features extensively in some works of Shkredov (see \cite{Sh} for example) and Schoen-Shkredov \cite{ScSh2}, and is already implicitly used in work of Sanders \cite[Lemma 1.9]{Sa4}, but its strength and utility in the current context was far from apparent before the work of Kelley and Meka. At a first reading of the following result the reader may wish to take $B_1=B_2=G$, in which case $\mu=\mu_{B_1}\circ \mu_{B_2}$ is just the usual uniform measure on $G$.

\begin{lemma}\label{lem-sift}
Let $p\geq 1$ be an integer and $\epsilon,\delta>0$. Let $B_1,B_2\subseteq G$, and let $\mu=\mu_{B_1}\circ\mu_{B_2}$. For any finite set $A\subseteq G$ with density $\alpha$, if 
\[S=\{ x\in G: \mu_A\circ \mu_A(x)>(1-\epsilon)\norm{\mu_A\circ \mu_A}_{p(\mu)}\},\]
then there are $A_1\subseteq B_1$ and $A_2\subseteq B_2$ such that 
\[\langle \mu_{A_1}\circ\mu_{A_2},\ind{S}\rangle \geq 1-\delta\]
and
\[\min\brac{\frac{\abs{A_1}}{\abs{B_1}},\frac{\abs{A_2}}{\abs{B_2}}}\gg \brac{\alpha\norm{\mu_A\circ \mu_A}_{p(\mu)}}^{2p+O_{\epsilon,\delta}(1)}.\]
\end{lemma}
Again, since we will apply this only in the case $\epsilon,\delta\gg 1$, we are not concerned with the behaviour of the $O_{\epsilon,\delta}(1)$ term, although we record here that the proof (which is identical to that in \cite{KM}) in fact allows for $O(\epsilon^{-1}\log(\delta^{-1}))$. We furthermore note that $A_i$ will take the form $B_i \cap (A+s_1) \cap \cdots \cap (A+s_p)$ for some randomly chosen shifts $s_j \in G$.

We shall prove Lemma \ref{lem-sift} after Lemma \ref{lem-drc} below, but first we note the following immediate special case, which is all we use when studying $\mathbb{F}_q^n$.

\begin{corollary}\label{cor-key}
Let $p\geq 1$ be an integer and $\epsilon>0$. If $A\subseteq G$ is such that $\norm{\mu_A\circ \mu_A}_p \geq 1+\epsilon$ and $S=\{ x : \mu_A\circ \mu_A(x)>1+\epsilon/2\}$, then there are $A_1,A_2\subseteq G$, both of density
\[\gg \alpha^{2p+O_\epsilon(1)},\]
such that
\[\langle \mu_{A_1}\circ \mu_{A_2},\ind{S}\rangle \geq 1-\epsilon/8.\]
\end{corollary}

We encode the dependent random choice argument underpinning Lemma \ref{lem-sift} as the following general lemma.

\begin{lemma}\label{lem-drc}
Let $p\geq 2$ be an even integer. Let $B_1,B_2\subseteq G$ and $\mu=\mu_{B_1}\circ\mu_{B_2}$. For any finite set $A\subseteq G$ with density $\alpha$ and function $f:G\to\bbr_{\geq 0}$ there exist $A_1\subseteq B_1$ and $A_2\subseteq B_2$ such that
\[\langle \mu_{A_1}\circ \mu_{A_2}, f\rangle\leq 2\frac{\langle (\mu_A\circ \mu_A)^p,f\rangle_\mu}{\norm{\mu_A\circ \mu_A}_{p(\mu)}^p}\]
and
\[\min\brac{\frac{\abs{A_1}}{\abs{B_1}},\frac{\abs{A_2}}{\abs{B_2}}}\geq \frac{1}{4}\alpha^{2p}\norm{\mu_A\circ \mu_A}_{p(\mu)}^{2p}.\]
\end{lemma}
\begin{proof}
For $s\in G^{p}$ let $A_1(s)=B_1\cap (A+s_1)\cap\cdots\cap (A+s_{p})$, and similarly for $A_2(s)$. Note that
\begin{align*}
\langle (\mu_A\circ \mu_A)^p, f\rangle_\mu
&=\bbe_{\substack{b_1\in B_1\\ b_2\in B_2}}\mu_A\circ \mu_A(b_1-b_2)^{p}f(b_1-b_2)\\
&=\bbe_{\substack{b_1\in B_1\\ b_2\in B_2}}\brac{\alpha^{-2}\bbe_{t\in G}\ind{A+t}(b_1)\ind{A+t}(b_2)}^{p}f(b_1-b_2)\\
&=\alpha^{-2p}\bbe_{\substack{b_1\in B_1\\ b_2\in B_2}}\bbe_{s\in G^p}\ind{A_1(s)}(b_1)\ind{A_2(s)}(b_2)f(b_1-b_2)\\
&=\tfrac{\alpha^{-2p}\abs{G}^2}{\abs{B_1}\abs{B_2}}\bbe_{s\in G^p}\langle \ind{A_1(s)}\circ \ind{A_2(s)}, f\rangle.
\end{align*}
In particular, applying this with $f\equiv 1$ we see that, if $\alpha_i(s)=\abs{A_i(s)}/\abs{G}$, then 
\[\norm{\mu_A\circ \mu_A}_{p(\mu)}^p=\tfrac{\alpha^{-2p}\abs{G}^2}{\abs{B_1}\abs{B_2}}\bbe_s \alpha_1(s)\alpha_2(s)\]
and 
\[\frac{\langle (\mu_A\circ \mu_A)^p,f\rangle_\mu}{\norm{\mu_A\circ \mu_A}_{p(\mu)}^p}=\frac{\bbe_{s}\langle \ind{A_1(s)}\circ \ind{A_2(s)}, f\rangle}{\bbe_s \alpha_1(s)\alpha_2(s)}=\eta,\]
say. Note that, if 
\[M=\tfrac{1}{2}\alpha^p(\abs{B_1}\abs{B_2}/\abs{G}^2)^{1/2}\norm{\mu_A\circ \mu_A}_{p(\mu)}^p,\]
then 
\begin{align*}
\bbe_s 1_{\alpha_1(s)\alpha_2(s)<M^2}\alpha_1(s)\alpha_2(s) 
&< M\brac{\bbe_s \bbe_{x\in G}1_{A_1(s)}(x)}^{1/2}\brac{\bbe_s \bbe_{x\in G}1_{A_2(s)}(x)}^{1/2}\\
&= M\alpha^p(\abs{B_1}\abs{B_2}/\abs{G}^2)^{1/2}=\frac{1}{2}\bbe_s \alpha_1(s)\alpha_2(s)
\end{align*}
and so
\[\bbe_s\langle \ind{A_1(s)}\circ \ind{A_2(s)},f\rangle =\eta \bbe \alpha_1(s)\alpha_2(s)\\
< 2\eta \bbe_s \alpha_1(s)\alpha_2(s)1_{\alpha_1(s)\alpha_2(s)\geq M^2}.\]
In particular there must exist some $s$ such that 
\[\langle \ind{A_1(s)}\circ \ind{A_2(s)},f\rangle< 2\eta \alpha_1(s)\alpha_2(s)1_{\alpha_1(s)\alpha_2(s)\geq M^2},\]
and the claim follows (note that the left-hand side is trivially $\geq 0$ and hence such an $s$ must satisfy $\alpha_1(s)\alpha_2(s)\geq M^2$). 
\end{proof}

The deduction of Lemma~\ref{lem-sift} is immediate from Lemma~\ref{lem-drc} with $f=\ind{G\backslash S}$. Indeed, by nesting of $L^p$ norms we can assume that $p$ is sufficiently large in terms of $\epsilon$ and $\delta$ (this is where the $O_{\epsilon,\delta}(1)$ term arises in the exponent), and that $p$ is an even integer. It then suffices to note that 
\[\langle \mu_{A_1}\circ \mu_{A_2},\ind{S}\rangle=1-\langle \mu_{A_1}\circ \mu_{A_2},\ind{G\backslash S}\rangle\]
and by definition of $S$ we have
\[\frac{\langle (\mu_A\circ \mu_A)^{p},\ind{G\backslash S}\rangle}{\norm{\mu_A\circ \mu_A}_{p}^p}\leq (1-\epsilon)^p\]
which is $\leq \delta/2$ if $p$ is large enough.

\section{The finite field case}\label{sec-ff}
In this section we prove Theorem~\ref{th-main-ff}, following the sketch of Section~\ref{sec-sketch}. Theorem~\ref{th-ff-gen} can proved in a very similar way. The following straightforward lemma is a form of Step 1.

\begin{lemma}\label{lem-ff1}
Let $\epsilon >0$. If $A,C\subseteq G$, where $C$ has density at least $\gamma$, then either
\begin{enumerate}
\item  $\abs{\langle \mu_A\ast \mu_A,\mu_C\rangle -1}\leq \epsilon$ or
\item $\norm{\mu_A\circ \mu_A-1}_p \geq \epsilon/2$ for some $p\ll\lo{\gamma}$.
\end{enumerate}
\end{lemma}
\begin{proof}
If the first alternative fails, then by H\"{o}lder's inequality, for any $p\geq 1$
\[\epsilon < \abs{ \langle \mu_A \ast \mu_A - 1, \mu_C \rangle } \leq \norm{\mu_A\ast \mu_A-1}_p\gamma^{-1/p}.\]
In particular, if we choose $p=2\lceil K\lo{\gamma}\rceil$ for some large constant $K$, then we deduce that
\[\norm{\mu_A\ast \mu_A-1}_{p}\geq \tfrac{1}{2}\epsilon.\]
It remains to note that, assuming without loss of generality that $p$ is an even integer,
\[\norm{\mu_A\ast \mu_A-1}_{p}^p=(\widehat{\mu_A}^2\ind{\neq \zeroGhat})^{(p)}(\zeroGhat)\leq(\Abs{\widehat{\mu_A}}^2\ind{\neq \zeroGhat})^{(p)}(\zeroGhat)=\norm{\mu_A\circ \mu_A-1}_p^p. \]
Again, although we have used a one-line Fourier proof here, this can also be seen using an entirely physical argument, which we sketch here. Note that $\mu_A\ast \mu_A-1=(\mu_A-1)\ast (\mu_A-1)$ and similarly for $\mu_A\circ \mu_A-1$. It suffices therefore to show that, for any function $f:G\to\bbc$, we have
\[\norm{f\ast f}_p^p\leq \norm{f\circ f}_p^p.\]
We can write the left-hand side as
\begin{align*}
\norm{f\ast f}_p^p 
&= \bbe_{x,y}\brac{\bbe_{u}f(x+u)f(y-u)}^p\\
&=\bbe_{u_1,\ldots,u_p}\brac{\bbe_x f(x+u_1)\cdots f(x+u_p)}\brac{\bbe_y f(y-u_1)\cdots f(y-u_p)}
\end{align*}
and the right-hand side as
\begin{align*}
\norm{f\circ f}_p^p 
&= \bbe_{x,y}\brac{\bbe_{u}f(x+u)\overline{f(y+u)}}^p\\
&=\bbe_{u_1,\ldots,u_p}\abs{\bbe_x f(x+u_1)\cdots f(x+u_p)}^2,
\end{align*}
and the desired inequality now follows from the Cauchy-Schwarz inequality.
\end{proof}

Steps 2 and 3 have already been proved as Lemma~\ref{lem-key-spec} and Corollary~\ref{cor-key}. For Step 4 we can use the following almost-periodicity result, which is \cite[Theorem 3.2]{SS}, as a black box.
\begin{theorem}[Almost-periodicity]\label{th-ap}
If $A_1,A_2,S\subseteq \bbf_q^n$ are such that $A_1$ and $A_2$ both have density at least $\alpha$, then there is a subspace $V$ of codimension
\[\mathrm{codim}(V)\ll_\epsilon \lo{\alpha}^4\]
such that
\[\abs{\langle \mu_V\ast \mu_{A_1}\ast \mu_{A_2},\ind{S}\rangle-\langle \mu_{A_1}\ast \mu_{A_2},\ind{S}\rangle}\leq \epsilon.\]
\end{theorem}
Importantly, note that no assumption is made on $S$, and there is no dependency on the density of $S$. We now complete the final step by combining everything thus far into a single density increment statement, which suffices for Theorem~\ref{th-main-ff} as discussed in Section~\ref{sec-sketch}.

\begin{proposition}\label{th-ip}
Let $\epsilon \in (0,1)$. If $A,C\subseteq \bbf_q^n$, where $C$ has density at least $\gamma$, then either
\begin{enumerate}
\item $\abs{\langle \mu_A\ast \mu_A,\mu_C\rangle -1}\leq \epsilon$ or
\item there is a subspace $V$ of codimension 
\[\ll_\epsilon \lo{\gamma}^4\lo{\alpha}^4\]
such that $\norm{\ind{A}\ast \mu_V}_\infty \geq (1+\Omega(\epsilon))\alpha$.
\end{enumerate}
\end{proposition}
\begin{proof}
By Lemma~\ref{lem-ff1}, if the first alternative fails, then $\norm{\mu_A\circ \mu_A-1}_p\geq \epsilon/2$ for some $p\ll \lo{\gamma}$. By Lemma~\ref{lem-key-spec} (applied to $f=\mu_A\circ \mu_A-1$) we deduce that $\norm{\mu_A\circ \mu_A}_p\geq 1+\epsilon/4$ for some $p\ll_\epsilon \lo{\gamma}$. Hence, by Corollary~\ref{cor-key}, there are $A_1,A_2$, both of density 
\[\geq \alpha^{O_\epsilon(\lo{\gamma})},\]
such that $\langle \mu_{A_1}\circ \mu_{A_2},\ind{S}\rangle \geq 1-\epsilon/32$, where $S=\{x : \mu_A\circ \mu_A(x)\geq 1+\epsilon/8\}$ . By Theorem~\ref{th-ap} there is a subspace $V$ of the required codimension such that
\[\langle \mu_V\ast \mu_{A_1}\circ \mu_{A_2},\ind{S}\rangle \geq 1-\tfrac{1}{16}\epsilon.\]
By definition of $S$, it follows that
\begin{align*}
1+\Omega(\epsilon)
&\leq (1+\epsilon/8)(1-\epsilon/16)\\
&\leq \langle \mu_V\ast \mu_{A_1}\circ \mu_{A_2},\mu_A\circ \mu_A\rangle\\
&\leq \norm{\mu_V\ast \mu_A}_\infty \norm{\mu_{A}\ast \mu_{A_2}\circ \mu_{A_1}}_1\\
&= \norm{\mu_V\ast \ind{A}}_\infty \alpha^{-1},
\end{align*}
and the proof is complete.
\end{proof}

\section{The integer case}\label{sec-int1}

A useful strategy in additive combinatorics is to first prove a result of interest over $\mathbb{F}_q^n$, making use of the abundancy of subspaces, and then translate this to a result over the integers, replacing various equalities with approximate equalities and subspaces with Bohr sets.

Bohr sets of low rank are analogues of subspaces of low codimension, and have played a central role in additive combinatorics since the work of Bourgain \cite{Bo}, with much of the theory further developed by Sanders \cite{SaRoth1,SaRoth2}. We have recalled the relevant definitions and properties in Appendix~\ref{app-bohr}.

In this section we will employ Steps 3, 4, and 5 of the Kelley-Meka approach, performed relative to Bohr sets, to prove the following technical statement, whose statement is convenient for the iterative proof. In the following section we will show how it, together with unbalancing and the regularity of Bohr sets, implies Theorem~\ref{th-int-gen}, and thence Theorems~\ref{th-main-int} and \ref{th-3A}.

We apologise for the daunting appearance and technicality of the statements in this and the next section; a certain overhead of notation and caveats is a sad fact of life when working with Bohr sets. The reader should be reassured, however, that all the essential ideas are as in the $\mathbb{F}_q^n$ case.

The approach taken here should be compared with that in \cite[Section 8]{KM} -- there Kelley and Meka also follow the $\bbf_q^n$ model, but instead use mixed analogues over multi-dimensional progressions and Bohr sets, instead of just Bohr sets as we do here. The two objects are, in a heuristic sense, identical, but the need to pass between them adds some complexity to the argument of Kelley and Meka.

\begin{theorem}\label{th-flat-on-bohr}
There is a constant $c>0$ such that the following holds. Let ${\epsilon,\delta\in (0,1)}$ and $p,k\geq 1$ be integers such that $(k,\abs{G})=1$. For any $A\subseteq G$ with density $\alpha$ there is a regular Bohr set $B$ with
\[ d=\rk(B) =O_{\epsilon}\left(\lo{\alpha}^5p^4\right) \quad\text{and}\quad \abs{B}\geq \exp\left(-O_{\epsilon,\delta}(\lo{\alpha}^6p^5\lo{\alpha/p})\right)\abs{G} \]
and some $A'\subseteq (A-x)\cap B$ for some $x \in G$ such that 
\begin{enumerate}
\item $\abs{A'}\geq (1-\epsilon)\alpha\abs{B}$,
\item $\abs{A'\cap B'}\geq (1-\epsilon)\alpha\abs{B'}$, where $B'=B_{\rho}$ is a regular Bohr set with ${\rho\in (\tfrac{1}{2},1)\cdot c\delta\alpha/d}$, and 
\item
\[\norm{\mu_{A'}\circ \mu_{A'}}_{p(\mu_{k\cdot B''}\circ \mu_{k\cdot B''}\ast \mu_{k\cdot B'''}\circ \mu_{k\cdot B'''})} <(1+ \epsilon)\mu(B)^{-1},\]
for any regular Bohr sets $B'' = B'_{\rho'}$ and $B'''=B''_{\rho''}$ satisfying ${\rho',\rho''\in(\frac{1}{2},1)\cdot c\delta\alpha/d}$.
\end{enumerate}
\end{theorem}

The proof of Theorem~\ref{th-flat-on-bohr} proceeds by repeated application of the following statement, which is suitable for iteration.

\begin{proposition}\label{prop-it}
There is a constant $c>0$ such that the following holds. Let $\epsilon>0$ and $p,k\geq 1$ be integers such that $(k,\abs{G})=1$. Let $B,B',B''\subseteq G$ be regular Bohr sets of rank $d$ such that $B''\subseteq B'_{c/d}$ and $A\subseteq B$ with relative density $\alpha$. If
    \[ \norm{ \mu_{A}\circ \mu_{A}}_{p(\mu_{k\cdot B'}\circ\mu_{k\cdot B'}\ast \mu_{k\cdot B''}\circ \mu_{k\cdot B''})} \geq \left(1+\epsilon\right) \mu(B)^{-1},\]
    then there is a regular Bohr set $B'''\subseteq B''$ of rank at most
    \[\rk(B''')\leq d+O_{\epsilon}(\lo{\alpha}^4p^4)\]
    and size
    \[\abs{B'''}\geq \exp(-O_{\epsilon}(dp\lo{\alpha/d}+\lo{\alpha}^5p^5))\abs{B''}\]
    such that
    \[ \norm{ \mu_{B'''}*\mu_A }_\infty \geq (1+c\epsilon)\mu(B)^{-1}. \]
\end{proposition}

We first explain how Theorem~\ref{th-flat-on-bohr} follows by iteration. In doing so we shall require a regularity `narrowing' trick originally due to Bourgain. The following form is \cite[Lemma 12.1]{BS}.

\begin{lemma}\label{lem-bour}
There is a constant $c>0$ such that the following holds. Let $B$ be a regular Bohr set of rank $d$, suppose $A\subseteq B$ has density $\alpha$, let $\epsilon>0$, and suppose $B',B''\subseteq B_\rho$ where $\rho\leq c\alpha\epsilon/d$. Then either
\begin{enumerate}
\item there is some translate $A'$ of $A$ such that $\abs{A'\cap B'}\geq (1-\epsilon)\alpha\abs{B'}$ and $\abs{A'\cap B''}\geq (1-\epsilon)\alpha \abs{B''}$, or 
\item $\norm{\ind{A}\ast \mu_{B'}}_\infty\geq (1+\epsilon/2)\alpha$, or
\item $\norm{\ind{A}\ast \mu_{B''}}_\infty \geq (1+\epsilon/2)\alpha$.
\end{enumerate}
\end{lemma}

\begin{proof}[Proof of Theorem~\ref{th-flat-on-bohr} assuming Proposition~\ref{prop-it}]
Let $C_{\epsilon}, D_{\epsilon,\delta} \geq 1$ be parameters to be specified later, and let $c$ be the smaller of $1/2$ and the constant $c$ in Proposition~\ref{prop-it}. Let $t\geq 0$ be maximal such that there is a sequence of regular Bohr sets, say $B^{(0)},\ldots,B^{(t)}$, and subsets of translates of $A$, say $A_0,\ldots,A_t$, such that the following holds:
\begin{enumerate}
\item $B^{(0)}=G$ and $A_0=A$,
\item each $B^{(i)}$ is a regular Bohr set of rank $d_i$ and
\[d_{i+1}\leq d_i + C_{\epsilon} \lo{\alpha}^4p^4\]
and
\[\Abs{B^{(i+1)}}\geq \exp(-D_{\epsilon,\delta}(dp\lo{\alpha/d}+\lo{\alpha}^5p^5))\Abs{B^{(i)}},\]
and
\item each $A_i$ is a subset of $B^{(i)}$ with density $\alpha_i$ such that $\alpha_{i+1}\geq (1+c\epsilon/4)\alpha_i$ for $0\leq i<t$. 
\end{enumerate}

Observe from point (3), and the trivial fact that $\alpha_i\leq 1$, that $t\ll_\epsilon \lo{\alpha}$. Note that this implies $d_t\ll_{\epsilon} \lo{\alpha}^5p^4$. 

We apply Lemma~\ref{lem-bour} with $c\epsilon/2$ in place of $\epsilon$, and $B=B^{(t)}$, $B'=B_{c'\alpha\epsilon/d_t}$ and $B''=B'_{c''\delta\alpha/d_t}$, where the constants are in particular chosen to ensure that $B',B''$ are both regular. Provided we pick $D_{\epsilon,\delta}$ large enough in terms of $C_{\epsilon}$, $\epsilon$, and $\delta$, Lemma~\ref{lemma:bohrsiz} and the maximality of $t$ ensure that we must be in the first alternative of Lemma~\ref{lem-bour}'s conclusion: there exists a translate $A_t-x$ such that $\abs{(A_t-x)\cap B'}\geq (1-c\epsilon/2)\alpha \abs{B'}$ and $\abs{(A_t-x)\cap B''}\geq (1-c\epsilon/2)\alpha\abs{B''}$. 

We claim that $A'=(A_t-x)\cap B'$, with the $B'$ and $B''$ above playing the role of $B$ and $B'$ respectively, satisfies the conclusions of Theorem~\ref{th-flat-on-bohr}. Indeed, the bounds on the rank and size of $B'$, and the density conditions on $A'$, are clearly satisfied.

Suppose for a contradiction that \[\norm{\mu_{A'}\circ \mu_{A'}}_{p(\mu_{k\cdot B'''}\circ \mu_{k\cdot B'''}\ast \mu_{k\cdot B''''}\circ \mu_{k\cdot B''''})}\geq (1+\epsilon)\mu(B')^{-1},\]
for some regular Bohr sets $B''' = B''_\rho$ and $B''''=B'''_{\rho'}$ satisfying $\rho,\rho'\in(\frac{1}{2},1)\cdot c\delta \alpha/d_t$.

The conditions of Proposition~\ref{prop-it} are met, and hence we deduce there is some $\tilde{B}\subseteq B''''$ of rank
\[\rk(\tilde{B})\leq \rk(B)+O_{\epsilon}(\lo{\alpha}^4p^4)\]
and
\[\Abs{\tilde{B}}\geq \exp(-O_{\epsilon}(d_tp\lo{\alpha/d_t}+\lo{\alpha}^5p^5))\Abs{B''}\]
and there is a translate of $A_t$, say $A_t-y$, such that 
\[\mu_{\tilde{B}}(A_t-y)\geq (1+c\epsilon)(1-c\epsilon/2)\alpha\geq (1+c\epsilon/4)\alpha,\]
say. This a contradiction to the maximality of $t$, provided $C_{\epsilon}$ matches the implicit constant in the first $O_{\epsilon}$-term, since we can take $B^{(t+1)}=\tilde{B}$ and $A_{t+1}=(A_t-y)\cap \tilde{B}$, noting that by Lemma~\ref{lemma:bohrsiz}
\[\abs{B''''}\geq \exp(-O_{\epsilon,\delta}(d\lo{\alpha/d_t)})\Abs{B^{(t)}}.\]
\end{proof}

Proposition~\ref{prop-it} is a consequence of Steps 3 and 4 (dependent random choice and almost-periodicity) of the Kelley-Meka approach. We will use the following version of almost-periodicity, which is essentially \cite[Theorem 5.4]{SS}.
\begin{theorem}[Almost-periodicity]\label{th-ap-int}
There is a constant $c>0$ such that the following holds. Let $\epsilon>0$ and $B,B'\subseteq G$ be regular Bohr sets of rank $d$. Suppose that $A_1\subseteq B$ with density $\alpha_1$ and $A_2$ is such that there exists $x$ with $A_2\subseteq B'-x$ with density $\alpha_2$. Let $S$ be any set with $\abs{S}\leq 2\abs{B}$. There is a regular Bohr set $B''\subseteq B'$ of rank at most
\[d+O_\epsilon(\lo{\alpha_1}^3\lo{\alpha_2})\]
and size
\[\abs{B''}\geq \exp(-O_\epsilon(d\lo{\alpha_1\alpha_2/d}+\lo{\alpha_1}^3\lo{\alpha_2}\lo{\alpha_1\alpha_2/d}))\abs{B'}\]
such that
\[\abs{\langle \mu_{B'}\ast \mu_{A_1}\circ \mu_{A_2},\ind{S}\rangle-\langle \mu_{A_1}\circ \mu_{A_2},\ind{S}\rangle}\leq \epsilon.\]
\end{theorem}
\begin{proof}
We apply \cite[Theorem 5.4]{SS} with the choices (with apologies for the unfortunate clash in variable naming between papers)
\[A\to -A_2,\quad M\to A_1,\quad L\to -S,\quad S\to B'_{c/d},\quad B\to B'_{c/d}\]
and note that
\[\Abs{-A_2+B'_{c/d}}\leq \Abs{B'+B'_{c/d}}\leq 2\abs{B'}\leq 2\alpha_2^{-1}\abs{A_2}\]
by regularity of $B'$. The statement now almost follows from \cite[Theorem 5.4]{SS} after observing that (in that theorem's language) we have $K\leq 2\alpha_2^{-1}$, $\sigma=1$, and $\eta\geq \alpha_1/2$, except that there the statement has a constraint on the rank and the width of $B''$, rather than the rank and the size. Nonetheless the width condition can be immediately converted into a lower bound for the size of $B''$ with Lemma~\ref{lemma:bohrsiz}.
\end{proof}

We will now use this almost-periodicity together with Lemma~\ref{lem-sift} to deduce the iterative step.
\begin{proof}[Proof of Proposition~\ref{prop-it}]
By averaging there exists some $x\in k\cdot B'+k\cdot B''$ such that
\[\norm{\mu_A\circ \mu_{A}}_{p(\mu_{k\cdot B'}\ast \mu_{k\cdot B''-x})}\geq (1+\epsilon)\mu(B)^{-1}.\]
We now apply Lemma~\ref{lem-sift} with $B_1=k\cdot B'$ and $B_2=k\cdot B''+x$. This produces some $A_1\subseteq k\cdot B'$ and $A_2\subseteq k\cdot B''-x$ such that, with $S=\{x : \mu_{A}\circ \mu_A(x)\geq (1+\epsilon/2)\mu(B)^{-1}\}$,
\[\langle \mu_{A_1}\circ \mu_{A_2},\ind{S}\rangle \geq 1-\epsilon/4\]
and
\[\min\brac{\frac{\abs{A_1}}{\abs{B'}},\frac{\abs{A_2}}{\abs{B''}}}\gg \alpha^{2p+O_{\epsilon}(1)}.\]
We now apply Theorem~\ref{th-ap-int} (with $k\cdot B'$ and $k\cdot B''$ playing the roles of $B$ and $B'$ respectively), noting that we can, without loss of generality, take $S$ to be supported in $A_1-A_2\subseteq k\cdot B'+k\cdot B''-x$ and so by regularity of $B'$
\[\abs{S}\leq \abs{B'+B''}\leq 2\abs{B'}.\]
This produces some $B'''\subseteq k\cdot B''$ of the required rank and size such that
\[\langle \mu_{B'''}\ast \mu_{A_1}\circ \mu_{A_2},\mu_A\circ \mu_A\rangle\geq (1+\epsilon/2)(1-\epsilon/4)(1-\epsilon/8)\mu(B)^{-1}\geq (1+c\epsilon)\mu(B)^{-1}\]
for some absolute constant $c>0$. The result now follows from averaging, since 
\[\langle \mu_{B'''}\ast \mu_{A_1}\circ \mu_{A_2},\mu_A\circ \mu_A\rangle\leq \norm{\mu_{B'''}\ast \mu_{A}}_\infty\norm{\mu_{A_2}\circ \mu_{A_1}\ast \mu_A}_1=\norm{\mu_{B'''}\ast \mu_{A}}_\infty.\]
\end{proof}

\section{Deduction of Theorems~\ref{th-int-gen}, \ref{th-main-int} and \ref{th-3A}}\label{sec-int2}

Finally, in this section we deduce Theorem~\ref{th-int-gen} from Theorem~\ref{th-flat-on-bohr} (using Step 2) and show (using Step 1) how it implies Theorems~\ref{th-main-int} and \ref{th-3A}. The following form of unbalancing relative to Bohr sets is sufficient.

\begin{proposition}\label{prop:Lp_orth}
There is a constant $c>0$ such that the following holds. Let $\epsilon >0$ and $p \geq 2$ be an integer. Let $B \subseteq G$ be a regular Bohr set and $A\subseteq B$ with relative density $\alpha$. Let $\nu : G \to \bbr_{\geq 0}$ be supported on $B_\rho$, where $\rho \leq c\epsilon\alpha/\rk(B)$, such that $\norm{\nu}_1=1$ and $\widehat{\nu}\geq 0$. If
    \[ \norm{(\mu_A-\mu_B) \circ (\mu_{A}-\mu_B)}_{p(\nu)} \geq \epsilon\, \mu(B)^{-1}, \]
    then there exists $p'\ll_\epsilon p$ such that
    \[ \norm{ \mu_{A}\circ \mu_{A}}_{p'(\nu)} \geq \left(1+\tfrac{1}{4}\epsilon\right) \mu(B)^{-1}. \]
\end{proposition}

\begin{proof}
Let us write
    \[ g =\mu_A\circ \mu_B + \mu_B \circ \mu_A - \mu_B\circ\mu_B \]
    and note that for any $p'\geq 1$
    \begin{align*}
        \norm{ \mu_{A}\circ \mu_{A}}_{p'(\nu)} &= \norm{ (\mu_A-\mu_B)\circ (\mu_A-\mu_B) + \mu(B)^{-1} + g - \mu(B)^{-1} }_{p'(\nu)} \\
        &\geq \norm{ (\mu_A-\mu_B)\circ (\mu_A-\mu_B) + \mu(B)^{-1} }_{p'(\nu)} - \norm{ g - \mu(B)^{-1} }_{p'(\nu)}.
    \end{align*}
    To bound the error term, note that $\mu_A\circ \mu_B(0) = \mu(B)^{-1}$ and that by regularity (for example with Lemma~\ref{lemma:regConv}), for $x \in \supp(\nu)$, 
    \begin{align*}
        \abs{\mu_A\circ\mu_B(x) - \mu_A\circ\mu_B(0)} &\leq \alpha^{-1} \mu(B)^{-1} \norm{\mu_B(\cdot+x) - \mu_B}_1 \\
        &\ll\rho \rk(B) \alpha^{-1} \mu(B)^{-1} \\
        &\leq \tfrac{1}{12} \epsilon \mu(B)^{-1},
    \end{align*}
    say, assuming $\rho$ is sufficiently small, and similarly for the other terms constituting $g$. Hence
    \[\norm{g-\mu(B)^{-1}}_{p'(\nu)}\leq \norm{g - \mu(B)^{-1}}_{L^\infty(\supp(\nu))} \leq \tfrac{1}{4}\epsilon\,\mu(B)^{-1}. \]
    It therefore suffices to find some $p'\ll_\epsilon p$ such that
    \[\norm{ (\mu_A-\mu_B)\circ (\mu_A-\mu_B) + \mu(B)^{-1} }_{p'(\nu)}\geq (1+\tfrac{1}{2}\epsilon)\mu(B)^{-1}.\]
This is immediate from Lemma~\ref{lem-key-spec} applied to $f=\mu(B)(\mu_A-\mu_B)\circ (\mu_A-\mu_B)$. 
\end{proof}

To deduce Theorem~\ref{th-int-gen} from Theorem~\ref{th-flat-on-bohr} we will need to pass from the measure $\mu_B$ to $\mu_{B'}\circ \mu_{B'}\ast \mu_{B''}\circ \mu_{B''}$. This is a technical issue with Bohr sets that does not arise in the $\bbf_q^n$ model case (note that these measures are identical when $B=B'=B''=G$).

\begin{proposition}\label{prop:pos_def_measures}
There is a constant $c>0$ such that the following holds. Let $p \geq 2$ be an even integer. Let $f : G \to \bbr$, let $B \subseteq G$ and $B', B'' \subseteq B_{c/\rk(B)}$ all be regular Bohr sets. Then
\[ \norm{ f\circ f }_{p(\mu_{B'}\circ \mu_{B'}\ast \mu_{B''}\circ \mu_{B''})} \geq \tfrac{1}{2} \norm{f*f}_{p(\mu_B)}. \]
\end{proposition}

\begin{proof}
    By an application of Lemma \ref{lemma:fourierbohr} with $L=4$, $\rho = c/4\rk(B)$ and $\nu=\mu_{B'}\ast \mu_{B'}\ast \mu_{B''}\ast\mu_{B''}$, we have
\begin{align*}
\mu_B 
&\leq 2\mu_{B_{1+4\rho}}*\nu.
\end{align*}
Hence
    \begin{align*}
        \norm{f*f}_{p(\mu_B)}^p &= \E_{x\in G}\, \mu_B(x)\, f*f(x)^p \\
        &\leq 2\,\E_{x\in G} \left(\mu_{B_{1+4\rho}}*\nu\right)(x) \, f*f(x)^p \\
        &= 2\, \E_{t \in B_{1+4\rho}} \E_{x\in G}\, \nu(x-t) \, f*f(x)^p.
        \end{align*}
By averaging there exists some $t$ such that
        \begin{align*}
            \norm{f*f}_{p(\mu_B)}^p &\leq 2\,\E_{x\in G}\, \nu(x-t) \, f*f(x)^p \\
            &= 2\sum_{\gamma \in \Ghat} \widehat{\nu}(\gamma)\gamma(-t) (\fhat^2)^{(p)}(\gamma) \\
            &\leq 2\sum_{\gamma \in \Ghat} \widehat{\nu}(\gamma) (\Abs{\fhat}^2)^{(p)}(\gamma) \\
            &= 2\norm{ f\circ f }_{p(\nu)}^p,
        \end{align*}
where we have used the fact that $\widehat{\nu}\geq 0$.
\end{proof}

\begin{proof}[Proof of Theorem~\ref{th-int-gen}]

We may, without loss of generality, assume that $\delta$ is sufficiently small in terms of $\epsilon$ and $k$. Let $p'\ll_\epsilon p$ satisfy the condition in Proposition~\ref{prop:Lp_orth}. Let $A'$ and $B$ be the regular Bohr sets provided by Theorem~\ref{th-flat-on-bohr} applied with $p$ replaced by $p'$ and $\epsilon$ replaced by $\epsilon/8$. We claim that this choice satisfies the conclusion of Theorem~\ref{th-int-gen}. It suffices to prove
\[\norm{(\mu_{A'}-\mu_B)\ast(\mu_{A'}-\mu_B)}_{p(\mu_{k\cdot B'})}\leq \epsilon \mu(B)^{-1}.\]
Suppose not. Let $B''=B'_{\rho}$ and $B'''=B''_{\rho'}$ be regular Bohr sets with $\rho=c_1/d$ and $\rho'=c_2/d$ for some sufficiently small constants $c_1,c_2>0$. By Proposition~\ref{prop:pos_def_measures} we have, if we let $f=\mu_{A'}-\mu_B$ for brevity,
\[\norm{f\ast f}_{p(\mu_{k\cdot B'})}\leq 2 \norm{f\circ f}_{p(\nu)}\]
where $\nu=\mu_{k\cdot B''}\circ \mu_{k\cdot B''}\ast \mu_{k\cdot B'''}\circ \mu_{k\cdot B'''}$. In particular, $\norm{f \circ f}_{p(\nu)}>\tfrac{1}{2}\epsilon\mu(B)^{-1}$, and so by Proposition~\ref{prop:Lp_orth} (noting that $\nu$ is supported on $k(B''+B''+B'''+B''')\subseteq B'_{4k\rho}\subseteq B_{c/d}$) we deduce that
\[\norm{\mu_{A'}\circ \mu_{A'}}_{p'(\nu)} \geq (1+\epsilon/8)\mu(B)^{-1},\]
which contradicts the conclusion of Theorem~\ref{th-flat-on-bohr}, and we are done. \end{proof}

Finally, to apply Theorem~\ref{th-int-gen} to three-term progressions and finding long arithmetic progressions in $A+A+A$, we record the following version of the H\"{o}lder lifting Step 1.

\begin{proposition}\label{prop-holder}
There is a constant $c>0$ such that the following holds. Let $\epsilon >0$. Let $B \subseteq G$ be a regular Bohr set and $A\subseteq B$ with relative density $\alpha$, and let $B' \subseteq B_{c\epsilon\alpha/\rk(B)}$ be a regular Bohr set and $C\subseteq B'$ with relative density $\gamma$. Either
\begin{enumerate}
\item $\abs{ \langle \mu_A*\mu_A, \mu_{C} \rangle - \mu(B)^{-1} } \leq \epsilon \mu(B)^{-1}$ or
\item there is some $p \ll\lo{\gamma}$ such that $\norm{ (\mu_A-\mu_B)*(\mu_A-\mu_B)}_{p(\mu_{B'})} \geq \tfrac{1}{2}\epsilon \mu(B)^{-1}$.
\end{enumerate}
\end{proposition}
\begin{proof}
We first note that 
\[\langle \mu_A\ast \mu_A,\mu_C\rangle = \langle (\mu_A-\mu_B)\ast(\mu_A-\mu_B),\mu_C\rangle+2\langle \mu_A\ast \mu_B,\mu_C\rangle -\langle \mu_B\ast \mu_B,\mu_C\rangle.\]
By the regularity of $B$ (more specifically Lemma~\ref{lemma:regConv}) and the fact that $C\subseteq B_{c\epsilon /\rk(B)}$, we have
\[\abs{\langle \mu_B\ast \mu_B,\mu_C\rangle-\mu(B)^{-1}}\leq \norm{\mu_B}_\infty\norm{\mu_B\ast \mu_C-\mu_B}_1\leq \tfrac{1}{8}\epsilon \mu(B)^{-1}.\]
Similarly, the fact that  $C\subseteq B_{c\epsilon \alpha/\rk(B)}$ implies that
\[\abs{\langle \mu_A\ast \mu_B,\mu_C\rangle-\mu(B)^{-1}}\leq \norm{\mu_A}_\infty\norm{\mu_B\ast \mu_C-\mu_B}_1\leq \tfrac{1}{16}\epsilon \mu(B)^{-1}.\]
It follows that, with $f=(\mu_A-\mu_B)\ast (\mu_A-\mu_B)$, we have
\[\abs{\langle \mu_A\ast \mu_A,\mu_C\rangle-\mu(B)^{-1}-\langle f,\mu_C\rangle}\leq \tfrac{1}{4}\epsilon\mu(B)^{-1}.\]
Therefore, if the first possibility fails, then 
\[\gamma^{-1}\abs{\langle f, \ind{C}\rangle_{\mu_{B'}}}=\abs{\langle f,\mu_C\rangle}\geq\tfrac{3}{4}\epsilon\mu(B)^{-1}.\]
By H\"{o}lder's inequality, for any $p\geq 1$, 
\[\norm{f}_{p(\mu_{B'})}\gamma^{1-1/p}\geq \abs{\langle f,\ind{C}\rangle_{\mu_{B'}}}.\]
We can choose some $p\ll \lo{\gamma}$ such that $\gamma^{-1/p}\leq 3/2$, and the proof is complete.
\end{proof}

\subsection{Three-term arithmetic progressions}

Theorem~\ref{th-main-int} is an immediate consequence of the following result that gives a lower bound for the number of three-term arithmetic progressions in an arbitrary set, coupled with the observation that if $A$ contains only trivial three-term arithmetic progressions then this count is at most $N$.

\begin{theorem}[Kelley-Meka]\label{th-main-int-count}
If $A\subseteq \{1,\ldots,N\}$ has size $\abs{A}=\alpha N$, then $A$ contains at least
\[\exp(-O(\lo{\alpha}^{12}))N^2\]
many three-term arithmetic progressions.
\end{theorem}
\begin{proof}
As usual, we begin by considering $A\subseteq\{1,\ldots,N\}$ as a subset of $G=\mathbb{Z}/(2N+1)\mathbb{Z}$ -- the density of this set within $G$ is still $\asymp \alpha$, and any three-term arithmetic progression in $A\subseteq G$ yields one in $A\subseteq \{1,\ldots,N\}$. 

We apply Theorem~\ref{th-int-gen} with $\epsilon=1/4$, $k=2$, and $p=\lceil K\lo{\alpha}\rceil$ for some large constant $K$. Let $A''=A'\cap B'$. If 
\[\langle \mu_{A'}\ast \mu_{A'},\mu_{2\cdot A''}\rangle \geq \tfrac{1}{2}\mu(B)^{-1}\]
we are done, since the left-hand side is at most $\ll\alpha^{-3}\mu(B)^{-2}\mu(B')^{-1}$ times the number of three-term arithmetic progressions in $A$, and by Lemma~\ref{lemma:bohrsiz} we have
\[\mu(B)\mu(B')\geq \exp(-O_\epsilon(\lo{\alpha}^{12})).\]
Otherwise, we are in the second case of Proposition~\ref{prop-holder}, which contradicts the conclusion of Theorem~\ref{th-int-gen}, and we are done. 
\end{proof}

\subsection{Arithmetic progressions in $A+A+A$}

As in the previous section, Theorem \ref{th-3A} follows immediately from the following statement for general groups, by embedding $\{1,\ldots,N\}$ in a cyclic group $\bbz/M\bbz$ for a prime $M$ between $2N$ and $4N$.

\begin{theorem}
    If $A \subseteq G$ has size $\alpha N$, then $A+A+A$ contains a translate of a Bohr set $B$ with
    \[ \rk(B) \ll \lo{\alpha}^{9} \quad\text{and}\quad \mu(B) \geq \exp(-O(\lo{\alpha}^{12}). \]
    In particular, if $G = \bbz/N\bbz$ for a prime $N$, then $A+A+A$ contains an arithmetic progression of length 
    \[\geq \exp(-O(\lo{\alpha}^3))N^{\Omega(1/\lo{\alpha}^{9})}.\]
\end{theorem}
\begin{proof}
We apply Theorem~\ref{th-int-gen} with $\epsilon=1/4$, $k=1$, and $p=\lceil K\lo{\alpha}\rceil$ for some large constant $K$. Let $B,B'$ be the Bohr sets produced by that conclusion, and $A' = (A-x) \cap B$ the corresponding restricted translate of $A$.

We first argue that $\abs{(A'+A')\cap B'}\geq (1-\alpha/4)\abs{B'}$. Indeed, otherwise if we let $C=B'\backslash (A'+A')$, then the first case of Proposition~\ref{prop-holder} is violated and the conclusion of Theorem~\ref{th-int-gen} means the second also cannot hold.

Let $B''=B'_{c\alpha/d}$, where $c>0$ is some small constant. We argue that $B''\subseteq A'+A'+A'$. If not, there is some $x\in B''$ such that $(A'+A')\cap (x-A')=\emptyset$, and so
\[\abs{A'\cap(B'-x)}=\abs{(x-A')\cap B'}\leq \abs{B'\backslash (A'+A')}\leq \tfrac{\alpha}{4}\abs{B'}.\]
By regularity, however, the left-hand side is at least 
\[\abs{A'\cap B'}-\abs{B\backslash (B'-x)}\geq \abs{A'\cap B'}-\tfrac{1}{4}\alpha \abs{B'}\geq \tfrac{\alpha}{2}\abs{B'},\]
which is a contradiction.

We have found some Bohr set $B''$ of rank $O(\lo{\alpha}^{9})$ and density 
\[\mu(B'')\geq \exp(-O(\lo{\alpha}^{12}))\]
such that $B''\subseteq A'+A'+A'$. It remains to note that $A'+A'+A'$ is contained in a translate of $A+A+A$ and to appeal to Lemma~\ref{lem-bohrap} to find an arithmetic progression in $B''$ of length
\[\geq \exp(-O(\lo{\alpha}^3))N^{\Omega(1/\lo{\alpha}^{9})}.\]
\end{proof}

\appendix

\section{Bohr sets}\label{app-bohr}
In abelian groups more general than $\bbf_q^n$, a useful substitute for genuine subgroups is the class of Bohr sets, introduced to additive combinatorics by Bourgain \cite{Bo}. Below we collect some standard facts about Bohr sets.

\begin{definition}[Bohr sets]
For a non-empty $\Gamma\subseteq \widehat{G}$ and $\nu\in [0,2]$ we define the Bohr set $B=\mathrm{Bohr}_\nu(\Gamma)$ as 
\[\mathrm{Bohr}_\nu(\Gamma)=\left\{ x\in G : \abs{1-\gamma(x)}\leq \nu\textrm{ for all }\gamma\in\Gamma\right\}.\]
We call $\Gamma$ the \emph{frequency set} of $B$ and $\nu$ the \emph{width}, and define the \emph{rank} of $B$ to be the size of $\Gamma$, denoted by $\rk(B)$. We note here that all Bohr sets are symmetric and contain $0$.

In fact, when we speak of a Bohr set we implicitly refer to the triple $(\Gamma,\nu,\mathrm{Bohr}_\nu(\Gamma))$, since the set $\mathrm{Bohr}_{\nu}(\Gamma)$ alone does not uniquely determine the frequency set nor the width. When we use subset notation, such as $B'\subseteq B$, this refers only to the set inclusion (and does not, in particular, imply any particular relation between the associated frequency sets or width functions). Furthermore, if $B=\mathrm{Bohr}_\nu(\Gamma)$ and $\rho\in(0,1]$, then we write $B_\rho$ for the same Bohr set with the width dilated by $\rho$, i.e. $\mathrm{Bohr}_{\rho\nu}(\Gamma)$, which is known as a \emph{dilate} of $B$.
\end{definition}

Bohr sets are, in general, not even approximately group-like, and may grow exponentially under addition. Bourgain \cite{Bo} observed that certain Bohr sets are approximately closed under addition in a weak sense which is suitable for our applications.

\begin{definition}[Regularity\footnote{The constant $100$ here is fairly arbitrary. Smaller constants are permissible.}]
A Bohr set $B$ of rank $d$ is regular if for all $\abs{\kappa}\leq 1/100d$ we have 
\[(1-100 d\abs{\kappa})\abs{B}\leq \abs{B_{1+\kappa}}\leq(1+ 100 d\abs{\kappa})\abs{B}.\]
\end{definition}

We record here the useful observation, frequently used in this paper, that if $(k,\abs{G})=1$ and $B$ is a regular Bohr set of rank $d$ then $k\cdot B$ is also a regular Bohr set of rank $d$ (and of course the same density), simply by replacing each character in the frequency set by an appropriate dilate.

For further introductory discussion of Bohr sets see, for example, \cite[Chapter 4]{TV}, in which the following basic lemmas are established.

\begin{lemma}\label{lemma:bohrreg}
For any Bohr set $B$ there exists $\rho\in[\tfrac{1}{2},1]$ such that $B_\rho$ is regular.
\end{lemma}

\begin{lemma}\label{lemma:bohrsiz}
If $\rho\in (0,1)$ and $B$ is a Bohr set of rank $d$, then $\abs{B_\rho}\geq (\rho/4)^d\abs{B}$.
\end{lemma}

The following standard lemmas indicate how regularity of Bohr sets will be exploited. The following is proved as, for example, \cite[Lemma 4.5]{BS}.

\begin{lemma}\label{lemma:regConv}
If $B$ is a regular Bohr set of rank $d$ and $\mu:G\to\bbr_{\geq 0}$ is supported on $B_\rho$, with $\rho \in (0,1)$, then
\[ \norm{ \mu_B*\mu - \mu_B }_{1} \ll \rho d\norm{\mu}_1. \]
\end{lemma}

The following is a minor generalization of, for example, \cite[Lemma 4.7]{BS}, which is stated with $\nu = \mu_{B'}^{(L)}$ for a subset $B' \subseteq B_\rho$; the proof is identical.

\begin{lemma}\label{lemma:fourierbohr}
There is a constant $c>0$ such that the following holds.  Let $B$ be a regular Bohr set of rank $d$ and $L\geq 1$ be any integer. If $\nu:G\to\bbr_{\geq 0}$ is supported on $L B_\rho$, where $\rho \leq c/Ld$, and $\norm{\nu}_1=1$, then 
\[\mu_B \leq 2\mu_{B_{1+L\rho}}\ast \nu.\]
\end{lemma}

Finally, we note the following simple lemma, which is useful for finding arithmetic progressions.

\begin{lemma}\label{lem-bohrap}
If $N$ is a prime and $B\subseteq \mathbb{Z}/N\mathbb{Z}$ is a Bohr set of rank $d$, then $B$ contains an arithmetic progression of length 
\[\gg \abs{B}^{1/d}.\]
\end{lemma}
\begin{proof}
Let $\rho=4(2/\abs{B})^{1/d}$, and note that by Lemma~\ref{lemma:bohrsiz} we have
\[\abs{B_\rho}\geq (\rho/4)^d\abs{B}=2.\]
In particular there exists some $x\in B_\rho\backslash \{0\}$. By the triangle inequality it is clear that $\{x,\ldots,\lfloor \rho^{-1}\rfloor x\}\subseteq B$, whence $B$ contains an arithmetic progression of length $\gg \rho^{-1}$.
\end{proof}

\end{document}